\documentclass[english]{smfart}
\usepackage[francais,english]{babel}
\usepackage{smfthm}
\usepackage{amssymb}
\usepackage{euscript}
\usepackage[all]{xypic}
\usepackage{url}

\newcommand{\Sh}{{\mathrm{Sh}}}
\newcommand{\cl}{{\mathrm{cl}}}
\newcommand{\rd}{{\mathrm{red}}}

\newcommand{\scr}[1]{\EuScript{#1}}
\newcommand{\Q}{\mathbb{Q}}
\newcommand{\Z}{\mathbb{Z}}
\newcommand{\R}{\mathbb{R}}
\newcommand{\A}{\mathbb{A}}
\newcommand{\C}{\mathbb{C}}
\newcommand{\OO}{\scr{O}}
\newcommand{\an}{\mathrm{an}}
\newcommand{\ip}[1]{\langle #1 \rangle}
\newcommand{\End}{\mathop{\mathrm{End}}\nolimits}
\newcommand{\Tate}{\underline{\mathrm{Tate}}}
\newcommand{\Sp}{\mathop{\mathrm{Sp}}\nolimits}
\renewcommand{\div}{\mathrm{div}}

\author{Nick Ramsey} \address{Department of Mathematics, 
  University of Michigan}
\email{naramsey@umich.edu} 
\title{The Overconvergent Shimura Lifting} 
\thanks{This research is supported in part by NSF Grant DMS-0503264}

\begin{document}
\frontmatter

\begin{abstract}
  We construct a rigid-analytic map from the the author's
  half-integral weight cuspidal eigencurve (see
  \cite{hieigencurve}) to its integral weight counterpart that
  interpolates the classical Shimura lifting.
\end{abstract}

\maketitle
\tableofcontents

\section{Introduction}

In \cite{shimura}, Shimura discovered the following remarkable
connection between holomorphic eigenforms of half-integral weight and
their integral weight counterparts.
\begin{theo}\label{classicallift}
  Let $F$ be a nonzero holomorphic cusp form of level $4N$, weight
  $k/2\geq 5/2$ and nebentypus $\chi$.  Assume that $F$ is an
  eigenform for $T_{\ell^2}$ and $U_{\ell^2}$ for all primes $\ell$
  with eigenvalues $\alpha_{\ell}$.  Then there exists a nonzero
  holomorphic cusp form $f$ of weight $k-1$, level $2N$, and character
  $\chi^2$ that is an eigenform for all $T_{\ell}$ and $U_{\ell}$ with
  eigenvalues $\alpha_{\ell}$.
\end{theo}
\noindent Strictly speaking, Shimura had only conjectured that $f$ is
of level $2N$, but this was proven shortly thereafter by Niwa in
\cite{niwa} for weights at least $7/2$ and then by Cipra for all weights
at least $5/2$ in \cite{cipra}.

Let $p$ be an odd prime and let $N$ be a positive integer with $p\nmid
N$.  In \cite{hieigencurve} the author constructed a rigid analytic
space $\widetilde{D}$ (denoted $\widetilde{D}^0$ there) parameterizing
all finite-slope systems of eigenvalues of Hecke operators acting on
overconvergent cuspidal $p$-adic modular forms of half-integral weight
and tame level $4N$.  Let $D$ denote the integral weight cuspidal
eigencurve of tame level $2N$ constructed, for example, in
\cite{buzzardeigenvarieties}.  In this paper we construct a
rigid-analytic map $\Sh:\widetilde{D}_\rd\longrightarrow D_\rd$ that
interpolates the Shimura lifting in the sense that if
$x\in\widetilde{D}$ is a system of eigenvalues occurring on a
\emph{classical} cusp form $F$ of half-integral weight (and such points are
shown to be Zariski-dense in $\widetilde{D}$) then $\Sh(x)$ is the
system of eigenvalues associated to the classical Shimura lifting of
$F$.

\section{Modular forms of  half-integral weight}

Fix an odd prime $p$ and let $\scr{W}$ denote $p$-adic weight space
over $\Q_p$.  We briefly recall a few facts and bits of notation
concerning $\scr{W}$.  See Section 2.4 of \cite{hieigencurve} for more
details.  The $K$-valued points of $\scr{W}$ (for a complete extension
$K/\Q_p$) correspond to continuous characters
$\kappa:\Z_p\longrightarrow K^\times$.  Each $\kappa\in\scr{W}(K)$
factors uniquely as $\kappa = \tau^i\cdot \kappa'$ where $\tau$ is the
Teichmuller character, $i$ is an integer well-defined modulo $p-1$,
and $\kappa'$ is trivial on $\mu_{p-1}\subseteq \Z_p^\times$.  The
space $\scr{W}$ is accordingly the admissible disjoint union of $p-1$
subspaces $\scr{W}^i$ for $0\leq i< p-1$.  Each $\scr{W}^i$ is
isomorphic to the open unit ball $B(1,1)$ around $1$ under the map
$$\kappa\longmapsto \kappa(1+p).$$
Also, $\scr{W}$ is the rising union
of the nested sequence of admissible open affinoids $\{\scr{W}_n\}$
whose points are those $\kappa$ with $|\kappa(1+p)^{p^{n-1}}-1|\leq
|p|$.  For an integer $i$ with $0\leq i< p-1$ and a positive integer
$n$ we set $\scr{W}^i_n = \scr{W}^i\cap \scr{W}_n$.  Finally, if
$\lambda\geq 0$ is an integer, we will usually denote the
$\Q_p$-valued point $x\longmapsto x^\lambda$ of $\scr{W}$ simply by
$\lambda$.

Let $N$ be a positive integer not divisible by the odd prime $p$.
Given a $p$-adic weight $\kappa\in\scr{W}(K)$, with $K$ and complete
and discretely-valued extension of $\Q_p$, and an $r\in[0,r_n]\cap \Q$,
we introduced in \cite{hieigencurve} the Banach space
$\widetilde{M}_\kappa(4N,K,p^{-r})$ of half-integral weight $p$-adic
modular forms of tame level $4N$ and weight $\kappa$ defined over $K$.
Here $\{r_n\}$ is the decreasing sequence of positive rational numbers
introduced in \cite{buzzardeigenvarieties} and \cite{hieigencurve},
the details of which will be of no importance to us in this paper.
This space is endowed with a continuous action of the Hecke operators
$T_{\ell^2}$ ($\ell\nmid 4Np$) and $U_{\ell^2}$ ($\ell\mid 4Np$), as
well as the tame diamond operators $\ip{d}_{4N}$, ($d\in
(\Z/4N\Z)^\times$).  A cuspidal subspace
$\widetilde{S}_\kappa(4N,K,p^{-r})$ is also defined, and is equipped
with all of the same operators.  In this section we will define a
space of classical modular forms of half-integral weight and use it
and results of \cite{hieigencurve} to define the classical subspaces of
$\widetilde{M}_\kappa(4N,K,p^{-r})$ and its cuspidal analog.  We will
then recast the classical Shimura lifting (Theorem
\ref{classicallift}) in these terms.
\begin{rema}
  The classical subspaces considered in this paper are limited in the
  sense that we restrict our attention to classical forms of level
  $4Np$.  One can also include classical forms of higher level $4Np^m$
  into the above spaces of $p$-adic forms.  We omit these forms here
  in part because we have no real need for them, and in part because
  we have not proven an analog of our control theorem (Theorem
  \ref{control}) for such forms (though we expect such a result to hold).
\end{rema}

For any positive integer $M$, let $\Sigma_{4M}$ be the $\Q$-divisor on
the algebraic curve $X_1(4M)_{\Q}$ given by $$\Sigma_{4M} =
\frac{1}{4}\pi^*[\mathbf{c}]$$
where $\mathbf{c}$ is the cusp on (the
coarse moduli scheme) $X_1(4)_{\Q}$ corresponding to the pair
$(\Tate(q),\zeta_4q_2)$ and
$$\pi:X_1(4M)_{\Q}\longrightarrow X_1(4)_{\Q}$$
is the natural map.
This divisor $\Sigma_{4M}$ is set up to look like the divisor of zeros
of the pullback of the Jacobi theta-function $\theta$ to $X_1(4M)_{\Q}$.
Indeed, if $F$ is a meromorphic function on $X_1(4M)^\an_\C$, then
$F\theta^k$ is a holomorphic modular form of weight $k/2$ if and only
if $\div(F)\geq -k\Sigma_{4M}$.  

Let $C_{4M}$ be the divisor on $X_1(4M)_{\Q}$ given by the sum of the
cusps at which $\Sigma_{4M}$ has integral coefficients (this includes,
in particular, all cusps outside of the support of $\Sigma_{4M}$).  If
$F$ is a meromorphic function on $X_1(4M)^\an_{\C}$, then $F\theta^k$
is a cuspidal modular form of weight $k/2$ if and only if $\div(F)\geq
-k\Sigma_{4M}+C_{4M}$.  The reason for omitting the cusps at which
$\Sigma_{4M}$ has non-integral coefficients is that, since $\div(F)$
has integral coefficients, $F\theta^k$ automatically vanishes at such
a cusp as soon as it is holomorphic there.

\begin{defi}
  Let $k$ be an odd positive integer.  The space of classical modular
  forms of weight $k/2$ and level $4M$ over $K$ is defined by
  $$\widetilde{M}_{k/2}^{\cl}(4M,K) =
  H^0(X_1(4M)^\an_K,\OO(k\Sigma_{4M})),$$
  and the subspace of cusp
  forms is defined by $$\widetilde{S}_{k/2}^{\cl}(4M,K) =
  H^0(X_1(4M)_K^\an,\OO(k\Sigma_{4M} - C_{4M})).$$
\end{defi}
\mbox{}\\
Both of these spaces are endowed with a geometrically defined action
of the Hecke operators $T_{\ell^2}$ ($\ell\nmid 4M$) and $U_{\ell^2}$
($\ell\mid 4M$) as well as the diamond operators $\ip{d}$ ($d\in
(\Z/4M\Z)^\times$).  The construction of the Hecke operators is a
``twisted'' version of the usual pull-back/push-forward through the
Hecke correspondence where one must multiply by a well-chosen rational
function (essentially the ratio of the pull-backs of $\theta^k$
through the maps defining the correspondence) on the source space of
the correspondence.  This construction is carried out in Section 6 of
\cite{mfhi} and Section 5 of \cite{hieigencurve} (where it is applied
to a slightly different space of forms).  The diamond operators are
simply given by pull-back with respect to the corresponding
automorphisms of $X_1(4M)_{\Q}$.

Suppose now that $M=Np$ with $p$ and odd prime not dividing $N$.  For
reasons of $p$-adic weight character book-keeping we separate the
diamond action into two kinds of diamond operators using the Chinese
remainder theorem.  For $d\in(\Z/p\Z)^\times$ we define $\ip{d}_p$ to
be $\ip{d'}$ where $d'$ is chosen so that $d'\equiv d\pmod{p}$ and
$d'\equiv 1\pmod{4N}$.  The operators $\ip{d}_{4N}$ for $d\in
(\Z/4N\Z)^\times$ are defined similarly, and for any $d$ prime to
$4Np$ there is a factorization $\ip{d} = \ip{d}_p\circ \ip{d}_{4N}$.

Let $k$ be an odd positive integer and define $\lambda = (k-1)/2$.  In
Section 6 of \cite{hieigencurve} we defined an injection
$$\widetilde{M}_{k/2}^{\cl}(4Np,K)^{\tau^j} \longrightarrow
\widetilde{M}_{\lambda\tau^j}(4N,K,p^{-r})$$
for any rational number
$r$ with $0\leq r\leq r_1$, where $()^{\tau^j}$
indicates the eigenspace of the $j^{\small\rm th}$ power of the
Teichmuller character $\tau$ for the action of the $\ip{d}_p$.  This
injection is equivariant with respect to all of the Hecke operators
and tame diamonds operators $\ip{}_{4N}$.  It is also compatible with
varying $r$ and in particular furnishes an injection
$$\widetilde{M}_{k/2}^{\cl}(4Np,K)^{\tau^j}\longrightarrow
\widetilde{M}_{\lambda\tau^j}^\dagger (4N,K)$$
into the space of all
overconvergent forms.  If we further restrict to the eigenspace of
a Dirichlet character $\chi$ modulo $4N$ (valued in $K$) for the
$\ip{d}_{4N}$, then we also get an embedding
$$\widetilde{M}_{k/2}^{\cl}(4Np,K,\chi\tau^j)\longrightarrow
\widetilde{M}^\dagger_{\lambda\tau^j}(4N,K,\chi)$$
of the space of
classical forms with nebentypus character $\chi\tau^j$ for the entire
group of diamond operators $(\Z/4Np\Z)^\times$ into the space of
overconvergent forms of tame nebentypus $\chi$.  The image of any of
these injections will be referred to as the \emph{classical} subspace
of the target.  Note that this definition is consistent with the
definition of a classical form given in \cite{hieigencurve}.
\begin{rema}
  Everything in the previous paragraph goes through verbatim when
  restricted to the respective spaces of cusp forms.  However, we
  caution that it is possible for a noncuspidal classical form to be
  mapped to the space of $p$-adic cusp forms
  $\widetilde{S}^\dagger_{\lambda\tau^j}(4N,K)$ under the above
  inclusions; the form need only vanish at the cusps in the connected
  component $X_1(4Np)^\an_{\geq 1}$ of the ordinary locus in
  $X_1(4Np)^\an_K$.  Still, by the \emph{classical subspace} of
  $\widetilde{S}_{\lambda\tau^j}^\dagger(4N,K)$ we mean the image of
  the map $$\widetilde{S}_{k/2}^{\cl}(4Np,K)^{\tau^j}\longrightarrow
  \widetilde{S}^\dagger_{\lambda\tau^j}(4N,K).$$
\end{rema}

\begin{theo}\label{control}
  Let $F$ be an element of
  $\widetilde{M}^\dagger_{\lambda\tau^j}(4N,K)$ or
  $\widetilde{S}^\dagger_{\lambda\tau^j}(4N,K)$  and suppose that
  there exists a monic polynomial $P(T)\in K[T]$ all of whose roots
  have valuation less than $2\lambda-1$ such that $P(U_{p^2})F=0$.
  Then $F$ is classical.
\end{theo}
\begin{proof}
  The case of $F\in\widetilde{M}^\dagger_{\lambda\tau^j}(4N,K)$ was
  settled in \cite{hieigencurve} (Theorem 6.1).  The proof follows
  Kassaei's approach in \cite{kassaei} and builds the classical form
  by analytic continuation and and gluing.  In particular, one writes
  down an explicit sequence of forms on the ``other'' component of the
  ordinary locus of $X_1(4Np)^\an_K$ that converges to the analytic
  continuation of $F$.  It is easy to see that these forms vanish at
  all cusps as long as $F$ does, so the proof of carries over to the
  cuspidal case verbatim.
\end{proof}

For later ease of use, we translate Theorem \ref{classicallift} into
the world of $p$-adic coefficients.
\begin{theo}\label{padicclassicallift}
  Suppose that $k\geq 5$ and
  $F\in\widetilde{S}^{\cl}_{k/2}(4Np,K,\chi\tau^j)$ is a nonzero
  classical eigenform for all Hecke operators $T_{\ell^2}$ and
  $U_{\ell^2}$ with eigenvalues $\alpha_{\ell}\in K$.  Then there
  exists a unique nonzero normalized classical cuspidal modular form
  $f$ of weight $k-1$, level $2Np$, and nebentypus $\tau^{2j}\chi^2$
  defined over $K$ that is an eigenform for all Hecke operators
  $T_{\ell}$ and $U_{\ell}$ with eigenvalues $\alpha_{\ell}$.
\end{theo}
\begin{proof}
  Fix an embedding $i:K\hookrightarrow \C$.  By $p$-adic GAGA, any
  $F\in\widetilde{M}_{k/2}^{\cl}(4Np,K)$ is the analytification of an
  element of $H^0(X_1(4Np)_K,\OO(\Sigma_{4Np}))$.  Pulling back via the
  embedding $i$ and passing to the complex analytic space we arrive at
  an element $F_i\in H^0(X_1(4Np)^\an_\C,\OO(\Sigma_{4Np}))$ depending
  on $i$.  The condition on the divisor of $F_i$ exactly guarantees
  that the meromorphic modular form $F_i\theta^k$ of weight $k/2$ is
  in fact holomorphic.  Moreover the association $F\longmapsto
  F_i\theta^k$ is equivariant for the action of Hecke and diamond
  operators on both sides.  This can be seen as a formal consequence
  of the (entirely parallel) construction of these operators on both
  spaces.  Alternatively, in case of the Hecke operators, this can be
  deduced by examining their effect on $q$-expansions.  Replacing the
  divisor $\Sigma_{4Np}$ with $\Sigma_{4Np} - C_{4Np}$ we see that the
  association $F\longmapsto F_i\theta^k$ also preserves the condition
  of cuspidality.

Suppose that $F\in\widetilde{S}_{k/2}^{\cl}(4Np,K)^{\tau^j}$ satisfies 
\begin{eqnarray*}
  T_{\ell^2}F &=& \alpha_{\ell}F\ \ \mbox{for all}\ \ \ell\nmid 4Np\\
  U_{\ell^2}F &=& \alpha_{\ell}F\ \ \mbox{for all}\ \ \ell\mid 4Np\\
  \ip{d}_{4N}F &=& \chi(d)F\ \ \mbox{for all}\ \ d\in(\Z/4N\Z)^\times
\end{eqnarray*}
for some Dirichlet character $\chi$ mod $4N$, with
$\alpha_{\ell},\chi(d)\in K$ for all $\ell$ and $d$.  It follows that
the holomorphic cusp form $F_i\theta^k$ is of weight $k/2$, level
$4Np$, nebentypus character $i\circ(\tau^j\chi)$, and is an eigenform
for all $T_{\ell^2}$ and $U_{\ell^2}$ with eigenvalues
$i(\alpha_{\ell})$.  By the classical lifting theorem (Theorem
\ref{classicallift}), we can associate to this form a cuspidal modular
form $f_i$ of weight $k-1$, level $2Np$, and nebentypus character
$i\circ(\tau^{2j}\chi^2)$ that is an eigenfunction for all $T_{\ell}$
and $U_{\ell}$ with eigenvalues $i(\alpha_{\ell})$.  By
complex-analytic GAGA, this form is actually an algebraic modular form
defined over $\C$ with all the same properties.  The $q$-expansion
coefficients of $f_i$ at the cusp $(\Tate(q),e^{2\pi i/4Np})$ are the
leading coefficient $a_1(f_i)$ times polynomials in the Hecke
eigenvalues $i(\alpha_{\ell})$.  Since $f_i\neq 0$, $a_1(f_i)\neq 0$
as well and we may normalize $f_i$ so that $a_1(f_i)=1$.  It now
follows from the $q$-expansion principle that $f_i$ is in fact an
algebraic modular form defined over the field $K$ of weight $k-1$,
level $4Np$, and nebentypus $\tau^{2j}\chi^2$ that is an eigenform for
all the $T_{\ell}$ and $U_{\ell}$ with eigenvalues $\alpha_{\ell}$.
Moreover, $f_i$ is completely determined by the eigenvalues
$\alpha_\ell$ for all $\ell$ and is therefore unique (and in
particular independent of $i$).
\end{proof}

\section{The eigencurves}

As the details of the construction of the relevant eigencurves will be
used extensively in the sequel, we briefly recall them here.  The
construction uses various Banach modules of modular forms equipped
with a Hecke action.  We refer the reader to Sections 6 and 7 of
\cite{buzzardeigenvarieties} for the integral weight definitions and
to Sections 4 and 5 of \cite{hieigencurve} for the half-integral
weight definitions.  We also refer the reader to
\cite{buzzardeigenvarieties} for foundational details concerning the
Fredholm theory that goes into the construction of eigenvarieties in
general.

For the moment, let $\scr{W}$ be any reduced rigid space over a
complete and discretely-valued extension field $K$ of $\Q_p$.  Fix a
set $\mathbf{T}$ with a distinguished element $\phi\in\mathbf{T}$.
Suppose that we are given, for each admissible affinoid open
$X\subseteq \scr{W}$, an $\OO(X)$-Banach module $M_X$ satisfying
property ({\it Pr}) of \cite{buzzardeigenvarieties}, equipped with map
\begin{eqnarray*}
  \mathbf{T} & \longrightarrow & \End_{\OO(X)}(M_X) \\
  t & \longmapsto & t_X
\end{eqnarray*}
whose image consists of commuting continuous endomorphisms and such
that $\phi_X$ is compact for each $X$.  Suppose also that for each
pair $X_1\subseteq X_2\subseteq \scr{W}$ of admissible affinoid opens we
are given a continuous injective map
$$\alpha_{12}:M_{X_1}\longrightarrow
M_{X_2}\widehat{\otimes}_{\OO(X_2)} \OO(X_1)$$
of $\OO(X_1)$-modules that
is a ``link'' in the sense of \cite{buzzardeigenvarieties}.  Finally,
suppose that these links commute with $\mathbf{T}$ in the sense that
$\alpha_{12}\circ t_{X_1} =
(t_{X_2}\widehat{\otimes}1)\circ\alpha_{12}$ for each $t\in\mathbf{T}$
and that they satisfy the cocycle condition
$\alpha_{13}=\alpha_{23}\circ\alpha_{12}$ for any triple $X_1\subseteq
X_2\subseteq X_3\subseteq\scr{W}$ of admissible open affinoids.

Out of this data one can use the machinery of
\cite{buzzardeigenvarieties} to construct rigid analytic spaces spaces
$D$ and $Z$ over $K$ called the \emph{eigenvariety} and \emph{spectral
  variety}, respectively, equipped with canonical maps
$$D\longrightarrow Z\longrightarrow \scr{W}.$$
The points of $D$
correspond to systems of eigenvalues of $\mathbf{T}$ acting on the
modules $\{M_X\}$ such that the $\phi$-eigenvalue is nonzero, in a
sense made precise below in Lemma \ref{eigenpoints}, and the map
$D\longrightarrow Z$ simply records the reciprocal of the
$\phi$-eigenvalue and a point in $\scr{W}$.

The space $Z$ is easy to define.  For any admissible affinoid
$X\subseteq\scr{W}$ we define $Z_X$ to be the zero locus of the
Fredholm determinant
$$P_X(T)=\det(1-\phi_X T\ |\ M_X)$$
in $X\times\A^1$.  The links
guarantee that this determinant is independent of $X$ in the sense
that if $X_1\subseteq X_2\subseteq\scr{W}$ are two admissible open
affinoids, then $P_{X_1}(T)$ is the image of $P_{X_2}(T)$ under the
natural restriction map on the coefficients.  It follows that we can
glue the $Z_X$ for varying $X$ covering $\scr{W}$ to obtain a space
$Z$ equipped with a map $Z\longrightarrow
\scr{W}$.

The construction of $D$ is more complicated, and involves first
finding a nice admissible cover of $Z$ and constructing the part of
$D$ over each piece separately and then gluing these pieces
together.  This cover is furnished by the following theorem (Theorem
4.6 of \cite{buzzardeigenvarieties}).
\begin{theo}\label{cover}
  Let $R$ be a reduced affinoid algebra over $K$, let $P(T)$ be a
  Fredholm series over $R$, and let $Z\subset \Sp(R)\times \A^1$
  denote the hypersurface cut out by $P(T)$ equipped with the
  projection $\pi: Z\longrightarrow \Sp(R)$.  Define $\scr{C}(Z)$ to
  be the collection of admissible affinoid opens $Y$ in $Z$ such that
  \begin{itemize}
  \item $Y'=\pi(Y)$ is an admissible affinoid open in $\Sp(R)$,
  \item $\pi: Y\longrightarrow Y'$ is finite, and
  \item there exists $e\in \OO(\pi^{-1}(Y'))$ such that $e^2=e$ and $Y$
  is the zero locus of $e$.
  \end{itemize}
  Then $\scr{C}(Z)$ is an admissible cover of $Z$.
\end{theo}
\noindent We will generally take $Y'$ to be connected in what follows.
This is not a serious restriction, since $Y$ is the disjoint union of
the parts lying over the various connected components of $Y'$.  We
also remark that the third of the above conditions follows from the
first two (this is observed in \cite{buzzardeigenvarieties} where
references to the proof are supplied).

Fix an admissible open affinoid $X\subseteq \scr{W}$ and fix
$Y\in\scr{C}(Z_X)$ with connected image $Y'\subseteq X$.  Let
\begin{equation}\label{defofpy}
  P_{Y'}(T) = \det(1-(\phi_X\widehat{\otimes}1)T\ |\
  M_X\widehat{\otimes}_{\OO(X)}\OO(Y'))
\end{equation}
Note that this is not in conflict with the existing notation $P_{Y'}$
(for an arbitrary connected admissible affinoid open
$Y'\subseteq\scr{W}$) by Lemma 2.13 of \cite{buzzardeigenvarieties}
and the above comments about the independence of $P_{X}$ on $X$.  When
an ambient $X\subseteq\scr{W}$ is fixed, we prefer to use the
definition (\ref{defofpy}) instead so as to avoid using the links.

As explained in Section 5 of \cite{buzzardeigenvarieties}, we can
associate to the choice of $Y$ a factorization $P_{Y'}(T)= Q(T)Q'(T)$
into relatively prime factors with constant term $1$, where $Q$ is a
polynomial of degree equal to the degree of the projection
$\pi:Y\longrightarrow Y'$ whose leading coefficient is a unit.
Geometrically speaking, $Y$ is the zero locus of the polynomial $Q$ in
$\pi^{-1}(Y')$ while its complement $\pi^{-1}(Y')\setminus Y$ is cut
out by the Fredholm series $Q'$.  By the Fredholm theory of
\cite{buzzardeigenvarieties} there is a unique decomposition
$$M_{X}\widehat{\otimes}_{\OO(X)}\OO(Y')\cong N\oplus F$$
into closed
$\phi$-invariant submodules with the property that $Q^*(\phi)$
vanishes on $N$ and is invertible on $F$.  Moreover, $N$ is projective
of rank equal to the degree of $Q$ and the characteristic power series
of $\phi$ on $N$ is $Q$.  The projectors onto to the submodules $N$
and $F$ are in the closure of $\OO(Y')[\phi]$, so by the commutativity
assumption these submodules are preserved under all of the
endomorphisms associated to elements of $\mathbf{T}$.  Let
$\mathbf{T}(Y)$ denote the $\OO(Y')$-subalgebra of $\End_{\OO(Y')}(N)$
generated by the endomorphisms $t_{X}\widehat{\otimes}1$ for
$t\in\mathbf{T}$.  This algebra is finite over $\OO(Y')$ and therefore
affinoid.

Since the polynomial $Q$ is the characteristic power series of $\phi$
on $N$ and has a unit for a leading coefficient, $\phi$ is invertible
on $N$.  Moreover, since $Q^*(\phi)=0$ on $N$, $Q(\phi^{-1})=0$ on $N$
as well.  Thus we have well-defined map
\begin{eqnarray*}
  \OO(Y)\cong \OO(Y')[T]/(Q(T)) & \longrightarrow & \mathbf{T}(Y) \\
  T & \longmapsto & \phi^{-1}
\end{eqnarray*}
which is to say that the affinoid $D_Y = \Sp(\mathbf{T}(Y))$ is
equipped with a natural finite map $D_Y\longrightarrow Y$.  These
affinoids and maps can be glued together for varying
$Y\in\scr{C}(Z_X)$ to obtain a space $D_X$ equipped with a map
$D_X\longrightarrow Z_X$.  Finally, using the links $\alpha_{ij}$ we
can glue over varying $X\subseteq \scr{W}$ to obtain a space $D$ and
canonical maps
$$D\longrightarrow Z\longrightarrow \scr{W}.$$

These spaces and maps have a particularly nice interpretation on the
level of points.  Let $L$ be a complete and discretely-valued
extension of $K$.
\begin{defi}
  A pair $(\kappa,\gamma)$ consisting of an $L$-valued point
  $\kappa\in \scr{W}(L)$ and a map of sets $\gamma:
  \mathbf{T}\longrightarrow L$ is an \emph{$L$-valued system of
    eigenvalues} of $\mathbf{T}$ acting on the $\{M_X\}$ if there
  exists an admissible open affinoid $X\subseteq \scr{W}$ containing
  $\kappa$ and a nonzero element $$m\in
  M_X\widehat{\otimes}_{\OO(X),\kappa} L$$
  such that
  $$(t\widehat{\otimes}1)m = \gamma(t)m$$
  for all $t\in\mathbf{T}$.
  This system of eigenvalues is called \emph{$\phi$-finite} if
  $\gamma(\phi)\neq 0$.
\end{defi}
Let $x$ be an $L$-valued point of $D$.  Then $x$ lives over a point
$\kappa_x$ in some admissible affinoid open $X\subseteq \scr{W}$
and moreover lies in $D_Y$ for some $Y\in\scr{C}(Z_X)$.
The associated $K$-algebra map 
$\mathbf{T}(Y)\longrightarrow L$
gives a map
$\gamma_x:\mathbf{T}\longrightarrow L$ of sets.
\begin{lemm}\label{eigenpoints}
  The association $x\longmapsto (\kappa_x,\gamma_x)$ is a well-defined
  bijection between the set of $L$-valued points of $D$ and the set of
  $L$-valued $\phi$-finite systems of eigenvalues of $\mathbf{T}$
  acting on the $\{M_X\}$.  The map $D\longrightarrow Z$ is given by 
  $$x\longmapsto (\kappa_x,\gamma_x(\phi)^{-1})$$
  on $L$-valued
  points.
\end{lemm}
\begin{proof}
  The first assertion is proven in \cite{buzzardeigenvarieties}.  The
  second is obvious from the the definition of the map
  $D\longrightarrow Z$.
\end{proof}

For the remainder of the paper, $p$ will denote an odd prime and
$\scr{W}$ will denote $p$-adic weight space over $\Q_p$.  Fix a
positive integer $N$ prime to $p$.  For each admissible affinoid open
$X\subseteq \scr{W}$ and each rational number $r\in [0,r_n]$, we
define $M_X(N,\Q_p,p^{-r})$ to be the $\OO(X)$-Banach module of
families of (integral weight) modular forms of tame level $N$ and
growth condition $p^{-r}$ on $X$.  This module has been defined, for
example, in \cite{colmaz} and \cite{buzzardeigenvarieties} where an
action of the Hecke operators $T_{\ell}$ ($\ell \nmid Np$) and
$U_{\ell}$ ($\ell\mid Np$) and the tame diamond operators $\ip{d}_{N}$
($d\in(\Z/N\Z)^\times$) is also defined.  These operators are
continuous and the operator $U_p$ is compact whenever $r>0$.
Similarly, we define $\widetilde{M}_X(4N,\Q_p,p^{-r})$ to be the
$\OO(X)$-Banach module of families of half-integral weight modular
forms of tame level $4N$ and growth condition $p^{-r}$ on $X$.  This
module was introduced in \cite{hieigencurve} where an action of the
Hecke operators $T_{\ell^2}$ ($\ell\nmid 4Np$) and $U_{\ell^2}$
($\ell\mid 4Np$) and the tame diamond operators $\ip{d}_{4N}$
($d\in(\Z/4N\Z)^\times$) is also defined.  These operators are
continuous and $U_{p^2}$ is compact whenever $r>0$.  Each of these
modules has a cuspidal submodule having all of the same operators and
properties and will be denoted by replacing the letter $M$ by the
letter $S$.  The tilde will be used throughout the paper to
distinguish half-integral weight objects from their integral weight
counterparts.


All of these modules of forms satisfy (\emph{Pr}) of
\cite{buzzardeigenvarieties}.  The system of Banach modules $\{S_X\}$,
where $S_X = S_X(N,\Q_p,p^{-r_n})$ and $n$ is chosen to be the
smallest integer such that $X\subseteq \scr{W}_n$, carries canonical
links defined in \cite{buzzardeigenvarieties} (strictly speaking on
the entire space of forms, but they can be defined in exactly the same
manner on the cuspidal submodule).  In \cite{hieigencurve} we
construct canonical links for the system $\{\widetilde{S}_X\}$ where
$\widetilde{S}_X =\widetilde{S}_X(4N,\Q_p,p^{-r_n})$.  The Hecke
operators and tame diamond operators in each case furnish commuting
endomorphisms that are compatible with these links, and $U_p$ and
$U_{p^2}$ are compact, so we may apply the above construction in both
the integral and half-integral weight cases.  Note that the
eigenvarieties and spectral varieties so obtained are equidimensional
of dimension $1$ by Lemma 5.8 of \cite{buzzardeigenvarieties}. We
denote the eigencurve and spectral curve associated to
$\{S_X(N,\Q_p,p^{-r_n})\}$ by $D$ and $Z$, respectively, and refer to
them as the \emph{cuspidal integral weight eigencurve} and
\emph{spectral curve} of tame level $N$, respectively.  Similarly we
denote the eigencurve and spectral curve associated to
$\{\widetilde{S}_X(4N,\Q_p,p^{-r_n})\}$ by $\widetilde{D}$ and
$\widetilde{Z}$, respectively, and refer to them as the \emph{cuspidal
  half-integral weight eigencurve} and \emph{spectral curve} of tame
level $4N$, respectively.
\begin{rema}
  In the interest of notational brevity we have chosen not to adorn
  these spacing so as to indicate that we are working only with the
  cuspidal part.  This puts the notation in conflict with, for
  example, the author's previous notation in \cite{hieigencurve}.
  Hopefully no confusion will arise from this conflict.
\end{rema}

\section{Some density results}\label{densityresults}

We will need some lemmas on the density of certain sets of classical
points in what follows.  
Following Chenevier (\cite{chenevier}), we call a subset
$\Sigma\subseteq \scr{W}(\C_p)$ \emph{very Zariski-dense} if for each
$\kappa\in \Sigma$ and each irreducible (equivalently, connected)
admissible affinoid open $V\subseteq\scr{W}$ containing $\kappa$,
$V(\C_p)\cap\Sigma$ is Zariski-dense in $V$.
\begin{lemm}\label{sigmaisvzd}
  The set $\Sigma$ of weights of the form $\lambda\tau^j$ for
  $\lambda\geq 0$ and $0\leq j< p-1$ is very Zariski-dense in
  $\scr{W}$.
\end{lemm}
\begin{proof}
  Fix $\lambda_0\tau^{j_0}\in \Sigma$ and suppose moreover that
  $\lambda_0\tau^{j_0}\in \scr{W}^i$.  Let $V$ be any connected
  admissible affinoid containing $\lambda_0$, so $V\subseteq
  \scr{W}^i$.  Recall that $\scr{W}^i$ is isomorphic to the open unit
  ball $B(1,1)$ about $1$, the isomorphism being
  $$\kappa\longmapsto \kappa(1+p).$$
  Note that the points in $\Sigma$
  are all $\Q_p$-valued.  Since $V(\Q_p)$ is open in the $p$-adic
  topology, there exists an $\epsilon > 0$ such that the closed ball
  $B(\lambda_0\tau^{j_0},\epsilon]$ of radius $\epsilon$ about
  $\lambda_0\tau^{j_0}$ is contained in $V$.  Admit for the moment
  that $B(\lambda_0\tau^{j_0},\epsilon](\Q_p)\cap \Sigma $ is Zariski
  dense in $B(\lambda_0\tau^{j_0},\epsilon]$.  We claim that this
  implies that $V(\Q_p)\cap \Sigma$ is Zariski-dense in $V$.  Let
  $T\subset V$ be an analytic subset containing $V(\Q_p)\cap \Sigma$.
  Then $T\cap B(\lambda_0\tau^{j_0},\epsilon]$ is an analytic subset
  of $B(\lambda_0\tau^{j_0},\epsilon]$ containing
  $B(\lambda_0\tau^{j_0},\epsilon](\Q_p)\cap \Sigma$ and by assumption
  we conclude that $T$ contains $B(\lambda_0\tau^{j_0},\epsilon]$.  It
  now follows from Lemma 2.2.3 of \cite{conradirredcpnts} that $T=V$.
  
  So we are reduced proving Zariski-density in the special case of
  $B(\lambda_0\tau^{j_0},\epsilon]$, which amounts to proving that
  this ball contains infinitely many elements of $\Sigma$.  But
  $$|\lambda\tau^j - \lambda_0\tau^{j_0}| = |(1+p)^\lambda -
  (1+p)^{\lambda_0}| = |(1+p)^{\lambda-\lambda_0}-1|$$
  and we can find
  infinitely many $\lambda\tau^j\in B(\lambda_0\tau^{j_0},\epsilon]$
  which make this as small as we like by choosing $j=j_0$ and
  $\lambda\equiv \lambda_0\pmod{ (p-1)p^N}$ for sufficiently large
  $N$.
\end{proof}

Let $X\subseteq \scr{W}$ be an admissible affinoid open.  For any
polynomial $\widetilde{h}$ over $\OO(X)$ in the Hecke operators
$T_{\ell^2}$, $U_{\ell^2}$ and the diamond operators $\ip{d}_{4N}$, we
denote by $\widetilde{Z}^{\widetilde{h}}_X$ the subspace of $X\times
\A^1$ defined as the zero locus of
$$\widetilde{P}^{\widetilde{h}}_X(T) = \det(1-\widetilde{h}U_{p^2}T\mid
\widetilde{S}_X)$$ in $X\times\A^1$.
\begin{defi}   Let $K/\Q_p$ be a finite extension.
\begin{itemize}
\item We call a $K$-valued point
  $(\kappa,\alpha)\in\widetilde{Z}_X^{\widetilde{h}}$ \emph{classical}
  if $\kappa\in\Sigma$ and there exists a nonzero classical form $F\in
  \widetilde{S}^{\dagger}_\kappa(4N,K)$ such that $\widetilde{h}U_{p^2}F
  = \alpha^{-1}F$.
\item Let $x\in \widetilde{D}(K)$ and let $(\kappa,\gamma)$ be the
  $K$-valued system of eigenvalues corresponding to $x$ by Lemma
  \ref{eigenpoints}.  Then $x$ is called \emph{classical} if
  $\kappa\in\Sigma$ and there exists a nonzero classical form $F\in
  \widetilde{S}^{\dagger}_{\kappa}(4N,K)$ on which the the operators
  $T_{\ell^2}$, $U_{\ell^2}$, and $\ip{d}_{4N}$ act through $\gamma$.
\end{itemize}
\end{defi}

\begin{lemm}\label{densitytwo}
  Let $X\subseteq \scr{W}$ be a connected admissible open affinoid
  containing an element of $\Sigma$.  Then the classical points are
  Zariski-dense in $\widetilde{Z}^{\widetilde{h}}_X$.
\end{lemm}
\begin{proof}
  Let $C$ be an irreducible component of
  $\widetilde{Z}_X^{\widetilde{h}}$.  By Theorem 4.2.2 of
  \cite{conradirredcpnts}, $C$ is a Fredholm hypersurface and hence
  has Zariski-open in image in $X$.  It follows that $C$ contains a
  point $x$ mapping to an element of $\Sigma$ by Lemma
  \ref{sigmaisvzd}.  Let $Y$ be an element of the canonical cover
  $\scr{C}(\widetilde{Z}^{\widetilde{h}}_X)$ containing $x$ with
  connected image $Y'\subseteq X$.  By Lemma \ref{sigmaisvzd},
  $\Sigma\cap Y'$ is Zariski-dense in $Y'$.  Moreover, the proof of
  the lemma shows that this is also true after omitting any finite
  collection of points in $\Sigma$.
  
  Let $N$ denote the direct summand of
  $\widetilde{S}_{X}\widehat{\otimes}_{\OO(X)}\OO(Y')$ corresponding to
  the choice of $Y$.  The $\OO(Y')$-module $N$ is projective of rank
  equal to the degree of $Y\longrightarrow Y'$ and is stable under the
  endomorphism $U_{p^2}$.  Moreover $U_{p^2}$ acts invertibly on $N$
  since $\widetilde{h}U_{p^2}$ does.  It follows that the eigenvalues
  of $U_{p^2}$ are bounded away from $0$ on $N$.  That is, there
  exists a positive integer $M$ such that for all finite extensions
  $K/\Q_p$ and all $K$-valued points $\OO(Y')\longrightarrow K$, the
  roots of the characteristic polynomial of $U_{p^2}$ acting on the
  fiber of $N\otimes_{\OO(Y')}K$ have absolute value at least $p^{-M}$.
  This can be seen, for example, by examining the variation of the
  Newton polygon of the characteristic polynomial
  \begin{equation}\label{charpolyonn}
    R(T)=\det(T-U_{p^2}\ |\ N)
  \end{equation}
  over $Y'$.  Let $\Sigma'$ denote the complement of the finite
  collection of $\lambda\tau^j\in \Sigma$ with $2\lambda -1 \leq M$,
  and note that $\Sigma'\cap Y'$ is still Zariski dense in $Y'$ by the
  above comments.  Since $Y\longrightarrow Y'$ is finite and flat,
  each irreducible component of $Y$ surjects onto $Y'$, and it follows
  that the preimage of $\Sigma'$ in $Y$ is Zariski-dense in every
  component of $Y$.  If $(\lambda\tau^j ,\alpha)$ is a $K$-valued
  point in this preimage, then there is a nonzero overconvergent form
  $$F\in N\otimes_{\OO(Y')}K\subseteq
  \widetilde{S}_{\lambda\tau^j}^\dagger(4N,K)$$
  with
  $\widetilde{h}U_{p^2}F=\alpha^{-1}F$.  Since $F$ is annihilated by
  $R_{\lambda\tau^j}(U_{p^2})$, the characteristic polynomial
  (\ref{charpolyonn}) with coefficients evaluated at $\lambda\tau^j$,
  it follows from Theorem \ref{control} that $F$ is a classical cusp
  form.  Thus the classical locus is also dense in each component of
  $Y$.  By Corollary 2.2.9 of \cite{conradirredcpnts}, $Y\cap C$ is a
  nonempty (since it contains $x$) union of irreducible components of
  $Y$.  It follows easily from Lemma 2.2.3 of \cite{conradirredcpnts}
  that a Zariski-dense subset of an admissible open in an irreducible
  space is in fact Zariski dense in the whole space, so we conclude
  that the classical points are Zariski-dense in $C$ for each $C$, and
  the lemma follows.
\end{proof}

\begin{rema}\label{densitythree}
  In case $\widetilde{h}=1$, this proof also shows that the set of
  points in $\widetilde{Z}_X$ of the form $(\lambda\tau^j, \alpha)$
  with $v(\alpha^{-1})< 2\lambda -1$ is Zariski-dense in
  $\widetilde{Z}_X$.  Since $\scr{W}$ is admissibly covered by a
  collection of connected admissible affinoid subdomains $\{X\}$ each
  of which meets $\Sigma$ (such as $\{\scr{W}^i_n\}$) we conclude that
  the set of $(\lambda\tau^j,\alpha)\in\widetilde{Z}$ with
  $v(\alpha^{-1})<2\lambda-1$ is Zariski-dense in all of
  $\widetilde{Z}$.
\end{rema}

\begin{coro}\label{densityfour}
  The classical points are Zariski-dense in $\widetilde{D}$.
\end{coro}
\begin{proof}
  By Lemma 5.8 of \cite{buzzardeigenvarieties}, the finite map
  $\widetilde{D}\longrightarrow \widetilde{Z}$ carries every
  irreducible component of $\widetilde{D}$ surjectively onto an
  irreducible component of $\widetilde{Z}$.  It follows from this and
  Remark \ref{densitythree} that the set of points $x\in
  \widetilde{D}$ such that the corresponding system of eigenvalues
  $(\kappa,\alpha)$ has $\kappa\in\Sigma$ and $v(\alpha^{-1})<
  2\lambda -1$ is Zariski-dense in $\widetilde{D}$.  But by Theorem
  \ref{control}, all such points are in fact classical.
\end{proof}

\section{Interpolation of the Shimura lifting}

Let $N$ be a positive integer not divisible by the odd prime $p$.  For
the remainder of the paper all spaces of modular forms of
half-integral weight will be taken at tame level $4N$ and all spaces
of modular forms of integral weight will be taken at tame level $2N$.
Let
$$\mathbf{2}:\scr{W}\longrightarrow\scr{W}$$
denote the finite map
given by $\mathbf{2}(\kappa)=\kappa^2$ on the level of points.  We
wish to construct maps $\Sh:\widetilde{D}_\rd \longrightarrow D_\rd$
and $\widetilde{Z}_\rd\longrightarrow Z_{\rd}$ fitting into the
diagram
\begin{equation}\label{padicshimura}
\xymatrix{ \widetilde{D}_\rd \ar[r]^{\Sh}\ar[d] & D_\rd\ar[d] \\
  \widetilde{Z}_\rd\ar[r]\ar[d] & Z_\rd\ar[d] \\ 
  \scr{W}\ar[r]^{\mathbf{2}} & \scr{W} }
\end{equation}
whose vertical arrows are the canonical ones arising from the
constructions of the eigencurves and such that if $x$ is a classical
point then $\Sh(x)$ is the system of eigenvalues of the classical Shimura
lift of a classical form corresponding to $x$.  The map
$\widetilde{Z}_\rd\longrightarrow Z_\rd$ will simply be given on
points by
$$(\kappa,\alpha)\longmapsto (\kappa^2,\alpha),$$
though it is not yet
at all clear that this map is well-defined.

We record a couple of lemmas concerning nilreductions of Fredholm
varieties that we will need in the sequel.

\begin{lemm}\label{coverreduction}
  With notation as in Theorem \ref{cover}, assume moreover that $R$ is
  relatively factorial.  Then the map
  \begin{eqnarray*}
    \scr{C}(Z) & \longrightarrow & \scr{C}(Z_\rd) \\
    Y & \longmapsto & Y_\rd
  \end{eqnarray*}
  is a bijection.
\end{lemm}
\begin{proof}
  First note that $Z_\rd$ is Fredholm by Theorem 4.2.2 of
  \cite{conradirredcpnts}.  Let $Y\in \scr{C}(Z)$ with image
  $Y'\subseteq \Sp(R)$.  That $Y\longrightarrow Y'$ is finite implies
  that $Y_\rd\longrightarrow Y'_\rd = Y'$ is finite.  To get an
  idempotent that cuts out $Y_\rd$, simply pull back the idempotent
  that cuts out $Y$ through the canonical reduction map,
  so the proposed map at least makes sense.
  
  Let $X\in \scr{C}(Z_\rd)$ with image $Y'\subseteq \Sp(R)$.  By the
  proof of A1.1 in \cite{conradmodrig}, the map $Z_\rd\longrightarrow
  Z$ is a homeomorphism of Grothendieck topologies.  In particular the
  underlying open set of $X$ is also an admissible open in $Z$.  As
  such, it inherits the structure of a rigid space by restricting the
  structure sheaf of $Z$ to $X$.  Let $Y$ denote the rigid space so
  obtained.  I claim that $Y\in \scr{C}(Z)$ and that the map
  $X\longmapsto Y$ is the inverse to the above map.
  
  Since reduction and passing to an admissible open commute, $Y_\rd =
  X$.  Now Theorem A1.1 of \cite{conradmodrig} implies that
  $Y\longrightarrow Y'$ is finite, so the comments following Theorem
  \ref{cover} imply that $Y\in \scr{C}(Z)$.  That these two maps are
  inverse to each other is clear. 
\end{proof}

If $F$ is a Fredholm series over a relatively factorial affinoid, let
$F_\rd$ denote the unique Fredholm series such that $Z(F_\rd) =
Z(F)_\rd$ (the existence and uniqueness of such a series is guaranteed
by Theorem 4.2.2 of \cite{conradirredcpnts}).
\begin{lemm}\label{reducedfred}
  Let $A$ and $B$ be relatively factorial affinoid algebras over $K$
  with $A$ an integral domain and let $f:\Sp(A)\longrightarrow \Sp(B)$
  be a map of rigid spaces over $K$.  Let $F$ and $G$ be Fredholm
  series over $A$ and $B$ respectively.  The following are equivalent.
  \begin{itemize}
  \item[(a)] $F_\rd$ divides $(f^*G)_\rd$ in the ring of entire series
    over $A$.
  \item[(b)] $f^*G$ vanishes on the zero locus of $F$.
  \item[(c)] $Z(F)_\rd$ is a union of irreducible components of
    $$(Z(G)\times_{\Sp(B)} \Sp(A))_\rd = Z(f^*G)_\rd.$$
  \item[(d)] there exists a unique map $$Z(F)_\rd\longrightarrow
    Z(G)_\rd$$
    that is given by $(x,\alpha)\longmapsto (f(x),\alpha)$
    on points.
  \end{itemize}
\end{lemm}
\begin{proof}
  Since passing to the reduction does not change the zero locus, (a)
  implies (b) trivially.  Suppose that (b) holds.  Then $(f^*G)_\rd$
  also vanishes on $Z(F)$, and (c) follows from Lemma 4.1.1 of
  \cite{conradirredcpnts}.  Now suppose that (c) holds.  By Corollary
  2.2.6 of \cite{conradirredcpnts}, $Z(F)_\rd$ is a ``component part''
  of $Z(f^*G)_\rd$ in the sense of \cite{colmaz}.  By Proposition
  1.3.4 of \cite{colmaz}, $F_\rd$ divides $(f^*G)_\rd$, and we have
  (a).  It remains to see that (d) is equivalent to these first three
  conditions. Suppose that that (c) holds.  The composite map
  $$Z(F)_\rd\hookrightarrow Z(f^*G)_\rd\longrightarrow
  Z(f^*G)=Z(G)\times_{\Sp(B)} \Sp(A)\longrightarrow Z(G)$$
  factors
  through a unique map $Z(F)_\rd\longrightarrow Z(G)_\rd$ by the
  universal property of reduction.  This map has the desired effect on
  points and is the unique one with this property since these spaces
  are reduced.  Finally, that (d) implies (b) is clear.
\end{proof}


\begin{prop}\label{prop:gwelldefined}
  Let $X$ and $\widetilde{X}$ be connected admissible affinoid
  subdomains of $\scr{W}$ such that $\mathbf{2}$ restricts to a map
  $\mathbf{2}:\widetilde{X}\longrightarrow X$, and suppose that
  $\widetilde{X}$ contains a point of $\Sigma$.  Let $h$ be a
  polynomial over $\OO(X)$ in the symbols $T_\ell$, $U_\ell$, and
  $\ip{d}_{2N}$ and let $\widetilde{h}$ denote the corresponding
  polynomial obtained by replacing these symbols by $T_{\ell^2}$,
  $U_{\ell^2}$, and $\ip{d^2}_{4N}$, respectively, and pulling back
  the coefficients to $\OO(\widetilde{X})$.  Let $$P_X^h(T) =
  \det(1-hU_p T\mid S_X)$$
  and
  $$\widetilde{P}_{\widetilde{X}}^{\widetilde{h}}(T) = \det(1-
  \widetilde{h}U_{p^2}T \mid\widetilde{S}_{\widetilde{X}}).$$
  Then
  $\widetilde{P}_{\widetilde{X}}^{\widetilde{h}}(T)_\rd$ divides
  $\mathbf{2}^*P_X^h(T)_\rd$.
\end{prop}
\begin{proof}
  By Lemma \ref{reducedfred} it suffices to check that
  $\mathbf{2}^*P_X^h(T)$ vanishes on the zero locus of
  $\widetilde{P}_{\widetilde{X}}^{\widetilde{h}}$.  By Lemma
  \ref{densitytwo}, the classical points are dense in this locus.  The
  same is true of the set of classical points of weight
  $\lambda\tau^j$ with $\lambda\geq 2$ since this omits only finitely
  many points in $\Sigma$.  Let $(\lambda\tau^j, \alpha)$ be a point
  in the zero locus of $\widetilde{P}_{\widetilde{X}}^{\widetilde{h}}$
  with $\lambda\geq 2$ and $\alpha$ in a finite extension $K$ of
  $\Q_p$ and define $k=2\lambda+1$.  Then the space
  $$V = \{F\in \widetilde{S}^{\cl}_{k/2}(4N,K)^{\tau^j}\ |\ 
  \widetilde{h}_{\lambda\tau^j}U_{p^2}F = \alpha^{-1}F\},$$
  where
  $\widetilde{h}_{\lambda\tau^j}$ denotes the polynomial
  $\widetilde{h}$ with coefficients evaluated at $\lambda\tau^j$, is
  nonzero.  Since all of the operators $T_{\ell^2}$, $U_{\ell^2}$, and
  $\ip{d}_{4N}$ commute with $\widetilde{h}_{\lambda\tau^j}U_{p^2}$,
  they act on the space $V$ and hence there exists a finite extension
  $L/K$ and a nonzero element $F\in
  \widetilde{S}^{\cl}_{k/2}(4N,L)^{\tau^j}$ that is a simultaneous
  eigenform for all of these operators.  In particular it has a
  character $\chi$ for the action of the $\ip{d}_{4N}$, and by
  Theorem \ref{padicclassicallift} there exists a nonzero classical
  form $f$ of level $2N$ and weight $k-1$ defined over $K$ that lies
  in the $\tau^{2j}$-eigenspace for $\ip{d}_p$ such that
  $$h_{(2\lambda)\tau^{2j}}U_p f = \alpha^{-1}f.$$
  Since $\alpha,\tau(d)\in K$ for all
  $d\in(\Z/p\Z)^\times$, there must also be a form $f$ defined over
  $K$ with these properties.  Thus $((2\lambda)\tau^{2j},\alpha)$ is a
  root of $P_X^h$, which is to say that $(\lambda\tau^j,\alpha)$ is a
  root of $\mathbf{2}^*P_X^h$.
\end{proof}

Let $X\subseteq \scr{W}$ be an admissible affinoid open and let
$\kappa\in \scr{W}(K)$ be a point.  For any module $M$ over $\OO(X)$ we
denote by $M_\kappa$ the vector space $M\otimes_{\OO(X),\kappa}K$ and
for any power series $P$ over $\OO(X)$ we denote by $P_\kappa$ the
power series over $K$ obtained by evaluating the coefficients of $P$
at $\kappa$.

\begin{coro}\label{divisibilitykappa}
  Let $\kappa\in\scr{W}(K)$, let $h_0$ be a polynomial over $K$ in the
  symbols $T_{\ell}$, $U_{\ell}$, and $\ip{d}_{2N}$, and let
  $\widetilde{h}_0$ denote the polynomial obtained by replacing these
  symbols by $T_{\ell^2}$, $U_{\ell^2}$, and $\ip{d^2}_{4N}$,
  respectively.  Pick a connected admissible affinoid open
  $\widetilde{X}$ containing $\kappa$ and let $X =
  \mathbf{2}(\widetilde{X})$.  Then $$\det(1 -
  \widetilde{h}_0U_{p^2}T\mid
  (\widetilde{S}_{\widetilde{X}})_{\kappa})_\rd\mid 
  \det(1 - h_0U_pT\mid (S_X)_{\kappa^2})_\rd$$
\end{coro}
\begin{proof}
  Note that $X$ is necessarily a connected admissible affinoid (and
  moreover $\mathbf{2}:\widetilde{X}\longrightarrow X$ is an
  isomorphism), so the assertion at least makes sense.  By enlarging
  $\widetilde{X}$ if necessary we may assume that $\widetilde{X}$
  contains an element of $\Sigma$, since the links $\alpha_{ij}$
  ensure that this enlargement does not affect the claimed
  divisibility.  Let $h$ be any polynomial over $\OO(X)$ in the symbols
  $T_{\ell}$, $U_{\ell}$, and $\ip{d}_{2N}$ with $h_{\kappa^2} =h_0$,
  and let $\widetilde{h}$ denote the polynomial over
  $\OO(\widetilde{X})$ obtained by replacing these symbols by
  $T_{\ell^2}$, $U_{\ell^2}$, and $\ip{d^2}_{4N}$, respectively, and
  pulling back the coefficients via $\mathbf{2}^*$.  Clearly we have
  $\widetilde{h}_{\kappa}=\widetilde{h}_0$.  By Proposition
  \ref{prop:gwelldefined} we have $$\det(1- \tilde{h}U_{p^2}T\mid
  \widetilde{S}_{\widetilde{X}})_\rd \mid
  (\mathbf{2}^*\det(1-hU_pT\mid S_X))_\rd.$$ 
  The result now follows from Lemma 2.13 of
  \cite{buzzardeigenvarieties} and Lemma \ref{reducedfred} by
  specializing to $\kappa$.
\end{proof}

\begin{coro}\label{gonx}
  Let $\widetilde{X}\subseteq\scr{W}$ be a connected admissible
  affinoid open and let $X=\mathbf{2}(\widetilde{X})$.  There is a
  unique finite map $\widetilde{Z}_{\widetilde{X},\rd}\longrightarrow
  Z_{X,\rd}$ having the effect $(\kappa,\alpha)\longmapsto
  (\kappa^2,\alpha)$ on points.
\end{coro}
\begin{proof}
  Choose integers $i$ and $n$ with $\widetilde{X}\subseteq
  \scr{W}^i_n$, so that $X\subseteq \scr{W}_n^{2i}$.  By Proposition
  \ref{prop:gwelldefined} (with $\widetilde{h}=1$) and Lemma
  \ref{reducedfred}, $\mathbf{2}^*P_{\scr{W}_n^i}(T)$ vanishes on the
  zero locus of $\widetilde{P}_{\scr{W}_n^{2i}}(T)$.  Let $\iota$ and
  $\widetilde{\iota}$ denote the inclusions of $X$ and $\widetilde{X}$
  into $\scr{W}^{2i}_n$ and $\scr{W}^i_n$, respectively.  Then 
  $$\widetilde{\iota}^*\mathbf{2}^*P_{\scr{W}_n^{2i}}(T) =
  \mathbf{2}^*\iota^*P_{\scr{W}_n^{2i}}(T)$$ vanishes on the zero
  locus of $\widetilde{\iota}^*P_{\scr{W}^i_n}(T)$.  Lemma 2.13
  and the links $\alpha_{ij}$ ensure that
  $$\iota^*P_{\scr{W}_n^{2i}}(T) = P_X(T)\ \ \ \mbox{and}\ \ \ 
  \widetilde{\iota}^*\widetilde{P}_{\scr{W}_n^i}(T) =
  \widetilde{P}_{\widetilde{X}}(T).$$
  The existence of a map having
  the indicated effect on points now follows from Lemma
  \ref{reducedfred}, and it remains to see that this map is finite.
  But by Lemma \ref{reducedfred} the map is the composition of the
  inclusion of the union of irreducible components
  $\widetilde{Z}_{\widetilde{X},\rd}$ into $(Z_X\times_X
  \widetilde{X})_\rd$ and the (nilreduction of) the projection
  $Z_X\times_X \widetilde{X}\longrightarrow Z_X$.  The former is
  obviously finite and the latter is finite because the map
  $\mathbf{2}:\widetilde{X} \longrightarrow X$ is an isomorphism.
\end{proof}

For each $i$, the $\scr{W}_n^i$ form a nested sequence, so we may glue
the maps furnished by Corollary \ref{gonx} applied to
$\widetilde{X}=\scr{W}^i_n$ and $X=\scr{W}_n^{2i}$ over increasing $n$
to obtain a diagram
$$\xymatrix{ \widetilde{Z}_{\scr{W}^i, \rd}\ar[r]\ar[d] &
  Z_{\scr{W}^{2i}, \rd}\ar[d] \\ \scr{W}^i \ar[r]^{\mathbf{2}} &
  \scr{W}^{2i}}$$
for each $0\leq i < p-1$.  Now since $\scr{W}$ is
the disjoint union of the $\scr{W}^i$, we obtain the bottom square in
the diagram (\ref{padicshimura}).  Let
$g:\widetilde{Z}_\rd\longrightarrow Z_\rd$ be the map so obtained.  On
the level of points, $g$ is simply given by $g(\kappa,\alpha) = (\kappa^2,
\alpha)$, as desired.  The next lemma shows that $g$ interacts well
with the canonical covers of its source and target.

\begin{lemm}\label{coverscompatible}
  Let $\widetilde{X}\subseteq \scr{W}$ be a connected admissible
  affinoid open and let $X = \mathbf{2}(\widetilde{X})$.  Let $Y$ be
  an element of the canonical cover $\scr{C}(Z_X)$ of $Z_X$ with
  connected image $Y'\subseteq X$.  Then $g^{-1}(Y_\rd)$ is either
  empty or is an element of
  $\scr{C}(\widetilde{Z}_{\widetilde{X},\rd})$ with connected image
  $\mathbf{2}^{-1}(Y')$.
\end{lemm}
\begin{proof}
  Since $g$ and $\mathbf{2}$ are finite, $g^{-1}(Y_\rd)$ and
  $\mathbf{2}^{-1}(Y')$ are affinoids. The latter is connected since
  $Y'$ is connected and $\mathbf{2}:\widetilde{X}\longrightarrow X$ is an
  isomorphism.
  
  By the construction of $g$ (see the proof of Lemma
  \ref{reducedfred}), $g^{-1}(Y_\rd)$ is the intersection inside
  $Z_{X,\rd}\times_X \widetilde{X}$ of the admissible affinoid
  $Y_\rd\times_X \widetilde{X}$ and the union of irreducible components
  $\widetilde{Z}_{\widetilde{X},\rd}$.  This intersection is a (possibly
  empty) union of irreducible components of the admissible affinoid
  $Y_\rd\times_X\widetilde{X}$ by Corollary 2.2.9 of
  \cite{conradirredcpnts}.  But $$Y_\rd\times_X \widetilde{X} \cong
  Y_\rd\times_{Y'} \mathbf{2}^{-1}(Y')$$
  is finite over
  $\mathbf{2}^{-1}(Y')$ since $Y_\rd$ is finite over $Y'$, and hence
  so is any nonempty subspace of components, such as $g^{-1}(Y_\rd)$
  (if nonempty).

  That $g^{-1}(Y_\rd)$ is disconnected from its complement in the full
  preimage of $\mathbf{2}^{-1}(Y')$ follows from the analogous property
  of $Y_\rd$; one simply pulls back the idempotent that cuts out $Y_\rd$
  though $g$ to get one that cuts out $g^{-1}(Y_\rd)$.
\end{proof}

Let $X$,$\widetilde{X}$, and $Y$ be as in the previous lemma and
assume that $g^{-1}(Y)\neq \emptyset$ so that $g^{-1}(Y_\rd)$ is in
$\scr{C}(\widetilde{Z}_{\widetilde{X},\rd})$ with connected image
$\mathbf{2}^{-1}(Y')$. 
  Let
$$P_{Y'}(T) = Q(T) Q'(T)$$
and
$$\widetilde{P}_{\mathbf{2}^{-1}(Y')}(T) = \widetilde{Q}(T)
\widetilde{Q}'(T)$$
denote the factorizations arising from the choice
of $Y$ and the $\widetilde{Y}\in
\scr{C}(\widetilde{Z}_{\widetilde{X}})$ of which $g^{-1}(Y_\rd)$ is
the underlying reduced affinoid (this well-defined by Lemma
\ref{coverreduction}).  
Since $g$ restricts to maps
$$g^{-1}(Y_\rd)\longrightarrow Y_\rd\ \ \ \mbox{and}\ \ \ 
\widetilde{Z}_{\mathbf{2}^{-1}(Y'),\rd}\setminus g^{-1}(Y_\rd)
\longrightarrow Z_{Y',\rd}\setminus Y_\rd,$$
Lemma \ref{reducedfred}
ensures that $\widetilde{Q}_\rd|(\mathbf{2}^*Q)_\rd$ and
$\widetilde{Q'}_\rd|(\mathbf{2}^*Q')_\rd$.  Let $A$ and $B$ denote the
affinoid algebras of $Y'$ and $\mathbf{2}^{-1}(S)$, respectively.  Let
$$S_{X}\widehat{\otimes}_{\OO(X)}A = N\oplus F$$
and
$$\widetilde{S}_{\OO(\widetilde{X})}\widehat{\otimes}B\cong \widetilde{N}
\oplus \widetilde{F}$$
denote the corresponding decompositions of the
spaces of families of cusp forms.

To construct the $p$-adic Shimura lift $\Sh$ and complete the diagram
(\ref{padicshimura}), we will construct the part covering
$g:g^{-1}(Y_\rd)\longrightarrow Y_\rd$ for each $Y$ and glue.  Let
$\mathbf{T}(Y)$ denote the $A$-subalgebra of $\End_R(N)$ generated by
$T_\ell$, $U_\ell$, and $\ip{d}_{2N}$ and let
$\widetilde{\mathbf{T}}(\widetilde{Y})$ denote the $B$-subalgebra of
$\End_S(\widetilde{N})$ generated by $T_{\ell^2}$, $U_{\ell^2}$, and
$\ip{d}_{4N}$.  We wish to show that the map
\begin{eqnarray}\label{shimuraonalgebras}
\mathbf{T}(Y)_\rd & \longrightarrow &
\widetilde{\mathbf{T}}(\widetilde{Y})_\rd \\ \nonumber
T_\ell & \longmapsto & T_{\ell^2} \\\nonumber
U_\ell & \longmapsto & U_{\ell^2} \\\nonumber
\ip{d}_{2N} & \longmapsto & \ip{d^2}_{4N}
\end{eqnarray}
that is given by $\mathbf{2}^*:A\longrightarrow B$ on coefficients, is
well-defined.  The following lemma furnishes the key divisibility
needed to prove this well-definedness.

\begin{lemm}\label{keydivisibility}
  Let $h$ be a polynomial over $A$ in the symbols $T_\ell$, $U_\ell$,
  and $\ip{d}_{2N}$, and let $\widetilde{h}$ be the polynomial over
  $B$ obtained by replacing these symbols by $T_{\ell^2}$,
  $U_{\ell^2}$, and $\ip{d^2}_{4N}$, respectively, and applying the
  map $\mathbf{2}^*:A\longrightarrow B$ to the coefficients.  Then
  $$\det(1- \widetilde{h}U_{p^2}T\mid \widetilde{N})_\rd \mid
  (\mathbf{2}^*\det(1- h   U_pT\mid N))_\rd.$$ 
\end{lemm}
\begin{proof}

  By Lemma \ref{reducedfred}, the divisibility claimed in the
  statement of the lemma can be checked after specializing to each
  $\kappa\in\mathbf{2}^{-1}(Y')$.  Define
  $$\widetilde{S}_\kappa = \widetilde{S}_{\widetilde{X}}
  \widehat{\otimes}_{\OO(\widetilde{X}),\kappa} K\ \ \mbox{and}\ \ 
  S_{\kappa^2} = S_{X}\widehat{\otimes}_{\OO(X),\kappa^2} K$$
  and define
  $\widetilde{N}_{\kappa}$ and $N_{\kappa^2}$ similarly.  By Lemma
  2.13 of \cite{buzzardeigenvarieties}, this amounts to checking that
  \begin{equation}\label{somedivisibility3}
    \det(1-\widetilde{h}_\kappa U_{p^2}T\mid \widetilde{N}_\kappa)_\rd |
    \det(1-h_{\kappa^2} U_p T\mid N_{\kappa^2})_\rd
  \end{equation}
  for each $\kappa\in \mathbf{2}^{-1}(Y')$.


  For any $\sigma\in \R$, we define $\widetilde{S}_\kappa^\sigma$
  ($S_{\kappa^2}^\sigma$, \dots, etc.) to be the slope $\sigma$
  subspace for the relevant operator ($U_p$ for integral weight spaces
  and $U_{p^2}$ for half-integral weight spaces).  These spaces are
  all finite-dimensional and $\widetilde{N}_\kappa$ is moreover the
  direct sum of the subspaces $\widetilde{N}_\kappa^\sigma$ that are
  nonzero since $U_{p^2}$ is invertible on $\widetilde{N}_{\kappa}$,
  and similarly for $N_{\kappa^2}$.  Let $h_0$ be any polynomial over
  $K$ in the symbols $T_\ell$, $U_\ell$, and $\ip{d}_{2N}$ and let
  $\widetilde{h}_0$ be the polynomial obtained by replacing the
  symbols by $T_{\ell^2}$, $U_{\ell^2}$, and $\ip{d^2}_{4N}$,
  respectively (no need to pull back the coefficients).  We claim that
  \begin{equation}\label{divisibilitybyslope}
    \det(T - \widetilde{h}_0\mid \widetilde{S}_\kappa^\sigma)_\rd |\
    \det(T - h_0 \mid S_{\kappa^2}^\sigma)_\rd 
  \end{equation}
  for all $\sigma$.  Since the endomorphisms of $N$ and
  $\widetilde{N}$ associated to $h_0$ and $\widetilde{h}_0$,
  respectively, are continuous, their eigenvalues are bounded and we
  can find a single nonzero $x\in K^\times$ such that both
  $h_0'=1+xh_0$ and $\widetilde{h}_0'=1+x\widetilde{h}_0$ have all
  eigenvalues of absolute value $1$.  It follows that
  $$\det(1-\widetilde{h}_0' U_{p^2} T\mid \widetilde{S}_\kappa)^\sigma =
  \det(1-\widetilde{h}_0'U_{p^2}T\mid \widetilde{S}_\kappa^\sigma)$$
  and
  $$\det(1- h_0' U_pT\mid S_{\kappa^2})^\sigma = \det(1-h_0' U_p T\mid
  S_{\kappa^2}^\sigma)$$
  where by $F(T)^\sigma$ for a Fredholm
  determinant $F(T)$ over $K$ we mean the unique (polynomial) Fredholm
  factor with the property that the Newton polygon of $F(T)^\sigma$
  has pure slope $\sigma$ and the Newton polygon of $F(T)/F(T)^\sigma$
  has no sides of slope $\sigma$.  But now it follows from
  Corollary \ref{divisibilitykappa} that
  $$\det(1-\widetilde{h}_0'U_{p^2}T\mid \widetilde{S}_\kappa^\sigma)_\rd |
  \det(1-h_0' U_p T\mid S_{\kappa^2}^\sigma)_\rd$$
  and since
  $\widetilde{h}_0'U_{p^2}$ and $h_0'U_p$ are invertible on these
  spaces, we can deduce the analogous divisibility for characteristic
  polynomials, $$\det(T-\widetilde{h}_0'U_{p^2}\mid
  \widetilde{S}_\kappa^\sigma)_\rd \mid \det(T-h_0'U_p|
  S_{\kappa^2}^\sigma)_\rd$$
  as well.  By the same reasoning, this
  divisibility holds after replacing $h_0'$ and $\widetilde{h}_0'$ by
  $h_0'+\epsilon$ and $\widetilde{h}_0'+\epsilon$ all sufficiently
  small $\epsilon\in K$.
  Let $$\det(T -
  (\widetilde{h}_0'+X)U_{p^2}\mid\widetilde{S}_\kappa^\sigma) = \prod_i
  (T - \widetilde{a}_iX - \widetilde{b}_i)$$
  and $$\det(T -
  (h_0'+X)U_p\mid S_{\kappa^2}^\sigma) = \prod_j (T - a_j X - b_j).$$
  Then
  we have shown that for infinitely many $\epsilon\in K$ it is the case
  that for each $i$ there exists $j$ such that
  $$\widetilde{a}_i\epsilon + \widetilde{b}_i = a_j\epsilon +b_j.$$
  It
  follows easily that for each $i$ there exists $j$ such that
  $\widetilde{a}_i=a_j$ and $\widetilde{b}_i = b_j$.  Simultaneously
  upper-triangularizing the commuting endomorphisms $\widetilde{h}_0'$
  and $U_{p^2}$, we see that the $\widetilde{b}_i/\widetilde{a}_i$ are
  exactly the eigenvalues of $\widetilde{h}_0'$ on
  $\widetilde{S}_\kappa^\sigma$.  Similarly, the $b_j/a_j$ are the
  eigenvalues of $h_0'$ on $S_{\kappa^2}^\sigma$, and we conclude that
  $$\det(T-\widetilde{h}_0'\mid\widetilde{S}_\kappa^\sigma)_\rd | \det(T
  - h_0'\mid S_{\kappa^2}^\sigma)_\rd.$$
  Now (\ref{divisibilitybyslope})
  follows from a linear change of variables.
  
  We claim that $$\det(1-\widetilde{h}_{\kappa} U_{p^2}\mid
  \widetilde{N}_\kappa^\sigma)_\rd \mid \det(1- h_{\kappa^2} U_p T \mid
  N_{\kappa^2}^\sigma)_\rd$$
  for each $\sigma\in \R$.  In particular,
  this establishes the desired divisibility (\ref{somedivisibility3})
  by the comments that follow it.  Let $\alpha$ be a root of
  $$\det(1-\widetilde{h}_{\kappa}U_{p^2}T\mid
  \widetilde{N}_{\kappa}^{\sigma}).$$
  By enlarging $K$ if necessary
  (the only requirement on $K$ in the preceding arguments is that it
  be finite over $\Q_p$ and contain the residue field of $\kappa$) we
  may assume that $\alpha\in K$.  Let $\scr{H}$ denote the free
  commutative $K$-algebra generated by the symbols $T_\ell$, $U_\ell$,
  and $\ip{d}_{2N}$.  The finite-dimensional $K$-vector space
  $S_{\kappa^2}^\sigma \oplus \widetilde{S}_\kappa^\sigma$ is a finite
  length algebra over $\scr{H}$, where $\scr{H}$ acts in the obvious
  way on $S_{\kappa^2}^\sigma$ and on $\widetilde{S}_\kappa^\sigma$ we
  agree that $T_\ell$ acts by $T_{\ell^2}$, $U_\ell$ acts by
  $U_{\ell^2}$, and $\ip{d}_{2N}$ acts by $\ip{d^2}_{4N}$.  Let
  $\widetilde{W}$ be a simple (over $\scr{H}$) constituent of $$\{ F
  \in \widetilde{N}_\kappa^\sigma\ |\ \widetilde{h}_\kappa U_{p^2} F =
  \alpha^{-1} F\}^{\mathrm{ss}} \neq 0,$$
  where for a (finite length)
  $\scr{H}$-module $\scr{M}$, $\scr{M}^{\mathrm{ss}}$ denotes the
  semisimplification of $\scr{M}$ as an $\scr{H}$-module.  By general
  facts about semisimple algebras, there exists $h_0\in\scr{H}$
  such that $h_0$ acts via the identity on $\widetilde{W}$ and $h_0=0$
  on any simple constituent of
  $S_{\kappa^2}^{\sigma,\mathrm{ss}}\oplus
  \widetilde{S}_{\kappa}^{\sigma,\mathrm{ss}}$ that is not isomorphic
  to $W$.  Divisibility (\ref{divisibilitybyslope}) implies that $1$
  is a root of $\det(T- h_0\mid S_{\kappa^2}^\sigma)$, so there must be a
  simple constituent $W$ of $S_{\kappa^2}^{\sigma,\mathrm{ss}}$
  isomorphic over $\scr{H}$ to $\widetilde{W}$.  I claim that moreover
  $W$ is a simple constituent of $N_{\kappa^2}^{\sigma,\mathrm{ss}}$.
  To see this, note that since $\widetilde{Q}_{\kappa,\rd} |
  Q_{\kappa^2,\rd}$ and these are polynomials, there exists a positive
  integer $M$ such that $\widetilde{Q}_\kappa | Q_{\kappa^2}^M$.  Thus
  the fact that $\widetilde{Q}_\kappa^*(U_{p^2})$ is zero on
  $\widetilde{N}_\kappa^\sigma$ (and therefore on $\widetilde{W}$)
  implies that $Q^*_{\kappa^2}(U_p)^M$ is zero on $W$.  But
  $Q_{\kappa^2}^*(U_p)$ is invertible on
  $S_{\kappa^2}^\sigma/N_{\kappa^2}^\sigma\cong F_{\kappa^2}^\sigma$,
  so $W$ must occur in $N_{\kappa^2}^{\sigma,\mathrm{ss}}$. Since $W$
  and $\widetilde{W}$ are isomorphic over $\scr{H}$, $h_{\kappa^2} U_p w =
  \alpha^{-1}w$ for all $w\in W\neq 0$, so $\alpha$ is a root of
  $$\det(1-hU_pT\mid N_{\kappa^2}^{\sigma,\mathrm{ss}}) = \det (1- hU_pT\mid
  N_{\kappa^2}^\sigma)$$ and the claimed divisibility follows.
\end{proof}

We now return to proving that the map (\ref{shimuraonalgebras}) is
well-defined.  In other words, we must show that if $h$ is a
polynomial over $A$ in the symbols $T_\ell$, $U_\ell$, and
$\ip{d}_{2N}$ that is nilpotent on $N$, and $\widetilde{h}$ is the
corresponding polynomial over $B$ as usual, then $\widetilde{h}$ is
nilpotent on $\widetilde{N}$.  But if $h$ is nilpotent on $N$ then
$\det(1-hU_p T\mid N) = 1$, so by Lemma \ref{keydivisibility},
$\det(1-\widetilde{h} U_{p^2}T\mid\widetilde{N})=1$ as well.  It
follows that $\widetilde{h} U_{p^2}$ is nilpotent on $\widetilde{N}$
since this module is projective of finite rank, and hence so is
$\widetilde{h}$ since $U_{p^2}$ acts invertibly on $\widetilde{N}$.

The (nonempty) admissible opens $g^{-1}(Y_\rd)$ form an admissible
cover of $\widetilde{D}_{\widetilde{X},\rd}$, so in order to glue the
maps we have defined to a map
$\widetilde{D}_{\widetilde{X},\rd}\longrightarrow D_{X,\rd}$ we must
check that they agree on the overlaps.  Since
$\mathbf{2}:\widetilde{X}\longrightarrow X$ is an isomorphism, this is
immediate from the characterization on points afforded by the
definition (\ref{shimuraonalgebras}) and Lemma \ref{eigenpoints} since
these spaces are reduced.  Now note that, for each $i$, the spaces
$\scr{W}^i$ and $\scr{W}^{2i}$ are covered by the nested families of
affinoids $\{\scr{W}_n^i\}$ and $\{\scr{W}_n^{2i}\}$, respectively,
and $\mathbf{2}:\scr{W}_n^i\longrightarrow \scr{W}_n^{2i}$ is an
isomorphism for each $n$.  Thus we may glue over increasing $n$ to
obtain a diagram
$$\xymatrix{ \widetilde{D}^i_\rd\ar[r]^{\Sh}\ar[d] & D^{2i}_\rd\ar[d] \\
  \widetilde{Z}^i_\rd \ar[r]^{g}\ar[d] & Z^{2i}_\rd \ar[d] \\
  \scr{W}^i\ar[r]^{\mathbf{2}} & \scr{W}^{2i}}$$
where the superscript
$i$ on a space mapping to $\scr{W}$ denotes the preimage of the
connected component $\scr{W}^i$.  Finally, since $\scr{W}$ is the
disjoint union of the $\scr{W}^i$ we obtain the desired diagram
(\ref{padicshimura}).

We now come to the main result of this paper.
\begin{theo}\label{maintheorem}
  Let $N$ be a positive integer and let $p$ be an odd prime not
  dividing $N$.  Let $D$ and $Z$ denote the integral weight cuspidal
  eigencurve and spectral curve of level $2N$, respectively, and let
  $\widetilde{D}$ and $\widetilde{Z}$ denote the half-integral weight
  cuspidal eigencurve and spectral curve of level $4N$, respectively.
  There exists a unique
  diagram
  $$\xymatrix{ \widetilde{D}_\rd \ar[r]^{\Sh}\ar[d] & D_\rd\ar[d] \\
    \widetilde{Z}_\rd\ar[r]^g\ar[d] & Z_\rd\ar[d] \\
    \scr{W}\ar[r]^{\mathbf{2}} & \scr{W} }$$
  where $\mathbf{2}$ and
  $g$ are characterized on points by $$\mathbf{2}(\kappa) = \kappa^2\ 
  \ \mbox{and}\ \ g(\kappa,\alpha) = (\kappa^2,\alpha)$$
  and $\Sh$ has
  the property that if $x\in \widetilde{D}(L)$ corresponds to a system
  of eigenvalues occurring on a nonzero classical form $F\in
  \widetilde{S}^{\cl}_{k/2}(4Np,K,\chi\tau^j)$ with $k\geq 5$, then
  $\Sh(x)$ corresponds to the system of eigenvalues associated to the
  image of $F$ under the classical Shimura lifting.
\end{theo}
\begin{proof}
  That the maps $\Sh$ and $g$ that we have constructed have the
  indicated properties is clear from their construction and Theorem
  \ref{padicclassicallift}.  Uniqueness follows from Lemma
  \ref{densityfour} since these spaces are reduced (the omission of the
  finitely many classical weights with $\lambda\leq 1$ does not affect
  the density of the classical points).
\end{proof}

\section{Properties of $\Sh$}

In this section we determine the nature of the image and fibers of the
map $\Sh$.  For each $i$, the map $\mathbf{2}:\scr{W}^i\longrightarrow
\scr{W}^{2i}$ is and isomorphism, and it follows easily from Lemma
\ref{eigenpoints} and the definition (\ref{shimuraonalgebras}) that
the restriction $\Sh: \widetilde{D}^i_\rd\longrightarrow D^{2i}_\rd$
is injective.
\begin{prop}
  The map $\Sh$ carries $\widetilde{D}^i_\rd$ isomorphically onto a union
  of irreducible components of $D^{2i}_\rd$.
\end{prop}
\begin{proof}
  Let $X\subseteq \scr{W}^{2i}$, $\widetilde{X}\subseteq \scr{W}^i$,
  and $Y$ be as in Lemma \ref{coverscompatible}.  Then by that lemma
  $g^{-1}(Y_\rd)$ is either empty or in
  $\scr{C}(\widetilde{Z}_{\widetilde{X},\rd})$.  Accordingly,
  $\Sh^{-1}(D(Y)_\rd)$ is either empty or equal to
  $\widetilde{D}(\widetilde{Y})_\rd$, where $\widetilde{Y}$ is the
  element of $\scr{C}(\widetilde{Z}_{\widetilde{X}})$ of which
  $g^{-1}(Y_\rd)$ is the underlying reduced affinoid.  In the latter
  case, the map
  $$\Sh^*:\mathbf{T}(Y)_\rd\longrightarrow
  \widetilde{\mathbf{T}}(\widetilde{Y})_\rd$$
  is clearly surjective
  from its definition (\ref{shimuraonalgebras}).  Thus the image of
  $\Sh$ is locally cut out by a coherent ideal (since the rings
  $\mathbf{T}(Y)$ are Noetherian), and is therefore an analytic set in
  $D^i_\rd$.  Since these spaces are reduced, $\Sh$ is an isomorphism
  of $\widetilde{D}^i_\rd$ onto this analytic set. Both
  $\widetilde{D}^i_\rd$ and $D^i_\rd$ are equidimensional of dimension
  one by Lemma 5.8 of \cite{buzzardeigenvarieties}, so Corollary 2.2.7
  of \cite{conradirredcpnts} ensures that the image of $\Sh$ is a
  union of irreducible components.
\end{proof}

Note that for each $i$, there are exactly two connected components of
$\scr{W}$ that map via $\mathbf{2}$ to $\scr{W}^{2i}$, namely
$\scr{W}^i$ and $\scr{W}^{i'}$ where $i'=i+(p-1)/2$.  We will
construct a canonical isomorphism
$\widetilde{D}^i\stackrel{\sim}{\longrightarrow}\widetilde{D}^{i'}$
fitting into a commutative diagram
\begin{equation}\label{upandsh}
\xymatrix{ \widetilde{D}^i_\rd\ar[drr]^{\Sh}\ar[dd]_{\sim} &
  & \\ & & D^{2i}_\rd \\ \widetilde{D}^{i'}_\rd\ar[urr]_{\Sh} & & }
\end{equation}
In particular, it will follows that the diagonal arrows have the same
union of irreducible components as image, and the map $\Sh$ is
everywhere two-to-one.  The existence of such an isomorphism stems
from the existence of a Hecke operator $U_p$ on families of
overconvergent forms that is a kind of ``square-root'' of the operator
$U_{p^2}$.

In \cite{mfhi} we constructed operators $U_{\ell}$ (on the spaces
considered in that paper) having the effect $\sum a_nq^n\longmapsto
\sum a_{\ell n} q^n$ on $q$-expansions, for all primes $\ell$ dividing
the level.  These operators were found to commute with all other Hecke
operators, but only commute with the diamond operators up to a factor
of the quadratic character $(\ell/\cdot)$.  In our case, if $\ell=p$,
then such a map would in fact alter the \emph{weight} since the
$p$-part of the nebentypus is part of the $p$-adic weight character.
Note that we have a factorization
$$\left(\frac{p}{\cdot}\right)=\left(\frac{-1}{\cdot}\right)^{(p-1)/2}
\left(\frac{(-1)^{(p-1)/2}p}{\cdot}\right) =
\left(\frac{-1}{\cdot}\right)^{(p-1)/2}\tau^{(p-1)/2}.$$
Let
$\epsilon$ denote the involution of $\scr{W}$ given by
$\epsilon(\kappa) = \kappa\cdot \tau^{(p-1)/2}$
\begin{prop}\label{up}
  Fix a primitive $(4Np)^{\small\rm th}$ root of unity
  $\zeta_{4Np}$.  Let $X\subseteq \scr{W}$ be a connected admissible
  affinoid open and let $r\in[0,r_n]\cap \Q$.  There is a compact
  $\OO(X)$-linear map
  $$U_p:\widetilde{M}_{\epsilon(X)}(4N,\Q_p,p^{-r})
  \widehat{\otimes}_{\OO(\epsilon(X))}\OO(X) \longrightarrow
  \widetilde{M}_{X}(4N,\Q_p,p^{-r})$$
  having the effect $\sum a_n
  q^n\longmapsto \sum a_{pn}q^n$ on $q$-expansions at
  $(\Tate(q),\zeta_{4Np})$.   This map commutes with the operators
  $T_{\ell^2}$ and $U_{\ell^2}$ for all $\ell$ and satisfies
  $$U_p\circ\ip{d}_{4N} =
  \left(\!\frac{-1}{d}\!\right)^{(p-1)/2}\ip{d}_{4N}\circ U_p.$$
\end{prop}
\begin{proof}
  The construction of $U_p$, like the all of the operators
  $T_{\ell^2}$ and $U_{\ell^2}$ follows the general procedure set up
  in Section 5 of \cite{hieigencurve}.  We will omit the details as
  they would take up considerable space and are very similar to the
  constructions of the operators $T_{\ell^2}$ and $U_{\ell^2}$, and
  content ourselves with commenting that the nontrivial commutation
  relation with the diamond operators arises (as it does in the
  construction of $U_{\ell}$ given in \cite{mfhi}) because the ``twisting''
  function $H$ on $X_1(4Np,p)^\an_{\Q_p}$ used in the construction
  is not fixed by the diamond operators as it is in the case of
  $T_{\ell^2}$ and $U_{\ell^2}$.
\end{proof}
Extending scalars to $\OO(\epsilon(X))$ and replacing $X$ by
$\epsilon(X)$ we arrive at a map in the opposite direction
$$\widetilde{M}_{X}(4N,\Q_p,p^{-r})\longrightarrow
\widetilde{M}_{\epsilon(X)}(4N,\Q_p,p^{-r})
\widehat{\otimes}_{\OO(\epsilon(X))}\OO(X)$$
that has the same effect on
$q$-expansions.  It follows that the composition of these maps in
either order is $U_{p^2}$ (or its scalar extension).  Everything we
have said about $U_p$ thus far holds equally well for cusps forms, and
Lemmas 2.12 and 2.13 of \cite{buzzardeigenvarieties} now imply that
\begin{eqnarray*}
\det(1-U_{p^2}T\mid \widetilde{S}_X) &=&
\det(1-(U_{p^2}\widehat{\otimes}1)T\mid
\widetilde{S}_{\epsilon(X)}\widehat{\otimes}_{\OO(\epsilon(X))}\OO(X))
\\ &=&
\epsilon^*\det(1-U_{p^2}T\mid \widetilde{S}_{\epsilon(X)}).
\end{eqnarray*}
Since
$\epsilon$ is an isomorphism we get a diagram
$$\xymatrix{ \widetilde{Z}_X\ar[r]^{\sim}\ar[d] &
  \widetilde{Z}_{\epsilon(X)}\ar[d] \\ X\ar[r]^{\epsilon} &
  \epsilon(X)}$$
in which the horizontal arrows square to the identity
map in the evident fashion.  Note that this diagram establishes a
bijection between the covers $\scr{C}(\widetilde{Z}_X)$ and
$\scr{C}(\widetilde{Z}_{\epsilon(X)})$.  Using the links, we see that
it is also compatible with enlarging $X$ and hence one obtains an
involution of the whole space $\widetilde{Z}$ covering the involution
$\epsilon$.

Let $Y\in\scr{C}(\widetilde{Z}_X)$ with connected image $Y'\subseteq
X$ and note that the corresponding element $Y_\epsilon\in
\scr{C}(\widetilde{X}_{\epsilon(X)})$ has connected image
$\epsilon(Y')\subseteq\epsilon(X)$.  Corresponding to $Y$ and
$Y_{\epsilon}$ we obtain decompositions
$$\widetilde{S}_{X}\widehat{\otimes}_{\OO(X)}\OO(Y')\cong
\widetilde{N}\oplus \widetilde{F}$$ and
$$\widetilde{S}_{\epsilon(X)}\widehat{\otimes}_{\OO(\epsilon(X))}\OO(\epsilon(Y'))
\cong \widetilde{N}_\epsilon\oplus \widetilde{F}_\epsilon$$
respectively.  Note that by extending scalars to $\OO(Y')$, $U_p$
induces a map
$$(\widetilde{S}_{\epsilon(X)}
\widehat{\otimes}_{\OO(\epsilon(X))}\OO(\epsilon(Y')))
\widehat{\otimes}_{\OO(\epsilon(Y'))}\OO(Y') \longrightarrow
\widetilde{S}_X \widehat{\otimes}_{\OO(X)}\OO(Y').$$
\begin{lemm}\label{uponn}
  This extension of scalars restricts to an isomorphism 
  $$U_p:\widetilde{N}_\epsilon \widehat{\otimes}_{\OO(\epsilon(Y'))}\OO(Y')
  \stackrel{\sim}{\longrightarrow} \widetilde{N}.$$
\end{lemm}
\begin{proof}
  Let $Q\in \OO(Y')[T]$ be the polynomial factor of
  $$\det(1-U_{p^2}T\mid \widetilde{S}_X
  \widehat{\otimes}_{\OO(X)}\OO(Y'))$$
  associated to the choice of $Y$,
  so that $\epsilon^*Q\in\OO(\epsilon(Y'))[T]$ is the polynomial
  associated to $Y_\epsilon$.  Thus the summand
  $\widetilde{N}_{\epsilon}\widehat{\otimes}_{\OO(\epsilon(Y'))}\OO(Y')$
  of $$(\widetilde{S}_{\epsilon(X)}\widehat{\otimes}_{\OO(\epsilon(X))}
  \OO(\epsilon(Y'))) \widehat{\otimes}_{\OO(\epsilon(Y'))}\OO(Y')$$
  is
  precisely the kernel of $$\epsilon^*(\epsilon^*Q^*(U_{p^2})) =
  Q^*(U_{p^2})$$
  and since the map $U_p$ commutes with the action of
  $U_{p^2}$, this summand maps to the kernel of $Q^*(U_{p^2})$ in
  $\widetilde{S}_X\widehat{\otimes}_{\OO(X)}\OO(Y')$ under $U_p$.  The
  latter is simply $\widetilde{N}$, so $U_p$ at least restricts to some
  map $$\widetilde{N}_{\epsilon}
  \widehat{\otimes}_{\OO(\epsilon(Y'))}\OO(Y')\longrightarrow
  \widetilde{N}.$$
  
  As above we can extend scalars to $\OO(\epsilon(Y'))$ and reverse the
  roles of $Y'$ and $\epsilon(Y')$ to get a map in the other direction
  with the property that both compositions are simply $U_{p^2}$ (or a
  scalar extension thereof) on the respective spaces.  That these maps
  are isomorphisms now follows from the fact that $U_{p^2}$ is
  invertible on the modules $\widetilde{N}$ and
  $\widetilde{N}_{\epsilon}$.
\end{proof}

It follows from Lemma \ref{uponn} and the commutation relations in
Proposition \ref{up} that the map
$$\widetilde{\mathbf{T}}(\widetilde{Y}_\epsilon)\longrightarrow
\widetilde{\mathbf{T}}(Y)$$
given by
$\epsilon^*:\OO(\epsilon(Y'))\longrightarrow \OO(Y')$ on coefficients
and 
\begin{eqnarray}\label{defofinv}
  T_{\ell^2} & \longmapsto & T_{\ell^2} \nonumber\\
  U_{\ell^2} &\longmapsto & U_{\ell^2} \\
  \ip{d}_{4N} & \longmapsto &
  \left(\!\frac{-1}{d}\!\right)^{(p-1)/2}\ip{d}_{4N} \nonumber
\end{eqnarray}
on the generators is in fact well-defined.  Thus we obtain a map
$\widetilde{D}_Y\longrightarrow \widetilde{D}_{Y_{\epsilon}}$ covering
the map $Y\longrightarrow Y_\epsilon$.  This construction readily
glues over $Y\in \scr{C}(\widetilde{Z}_X)$.  Moreover, the maps
$U_{p}$ are compatible with the canonical links used in the
construction of $\widetilde{D}$ (because these links are simply
induced by restriction to smaller admissible opens in
$X_1(4Np)^\an_K$), and it follows that these maps glue to a map
$\widetilde{D}^i\longrightarrow \widetilde{D}^{i'}$.  When this map is
composed with the one obtained by reversing the roles of $i$ and $i'$
(in either direction), one obtains the identity, as is evident from
the definition (\ref{defofinv}).  In particular it is an isomorphism
and extends to an involution of the whole space $\widetilde{D}$.
Finally, that the diagram (\ref{upandsh}) commutes can be checked on
points since these spaces are reduced.  But then it follows
immediately from the characterization of these points in terms of
systems of eigenvalues in Lemma \ref{eigenpoints} and the definition
(\ref{defofinv})

\begin{exem}
The restriction $k\geq 5$ in Theorem \ref{classicallift} is there to
avoid certain theta functions that show up in weights $1/2$ and $3/2$.
There is no meaningful modular lifting of the theta series of weight
$1/2$ (but see \cite{cipra}).  However, one \emph{can} lift cuspidal
theta functions of weight $3/2$, but one obtains Eisenstein series of
weight $2$ instead of cusp forms.  Let $\psi$ be a primitive Dirichlet
character modulo a positive integer $r$ such that $\psi(-1)=-1$.  Then
$$\theta_\psi(q) = \frac{1}{2}\sum_{n\in\Z}\psi(n)nq^{n^2}$$
is a
classical cusp form of weight $3/2$, level $4r^2$ and nebentypus
character $\psi_1 = \psi\cdot(-1/\cdot)$.  The form $\theta_\psi$ is
an eigenform for all $T_{\ell^2}$ ($\ell\nmid 4r^2$) and $U_{\ell^2}$
($\ell\mid 4r^2)$ with eigenvalues $(1+\ell)\psi(\ell)$.  The Shimura
lift of this form is the Eisenstein series $$E_\psi(q) =
\sum_{n=1}^\infty \psi(n)\sigma(n)q^n$$
of weight $2$, where
$\sigma(n) = \sum_{d\mid n}d$, as it easy to check from the explicit
formulas in \cite{shimura}.

Let $p$ be an odd prime.  If $p\mid r$, then $\theta_\psi$ is in the
kernel of $U_{p^2}$ and does not furnish a point on the eigencurve
$\widetilde{D}$.  If on the other hand $p\nmid r$, then $\theta_\psi$
thought of in level $4r^2p$ is in fact a $U_{p^2}$ eigenform with
eigenvalue $\psi(p)p\neq 0$.  Thus $\theta_\psi$ furnishes a point on
$\widetilde{D}$ and the existence of the map $\Sh$ implies that there
exists a \emph{cuspidal} $p$-adic eigenform with the same eigenvalues.
It is easy to check that the form $$E_\psi^* = E_\psi -
\psi(p)V_pE_\psi$$
has the expected eigenvalues, so this must be the
image form.  In fact it is easy to check using the explicit formulas
in \cite{shimura} that this is in fact the classical Shimura lift
applied to $\theta_\psi$ thought of at level $4r^2p$.

Thus we a lead to the conclusion that $E^*_\psi$, while not a
classical cusp form, is a cuspidal $p$-adic modular form.  That is,
$E^*_\psi$ vanishes at the cusps in the connected component
$X_1(4r^2p)^\an_{\geq 1}$ of the ordinary locus.  This fact also
follows from Theorem 3.4 of \cite{cipra}, where Cipra computes the
value of $E^*_\psi$ at every cusp.  We also remark that this is only
possible since $E_\psi^*$ is of critical slope (in this case, slope
$1$) since if it were of low slope then the technique of Kassaei
\cite{kassaei} implies that a low-slope form that is $p$-adically
cuspidal is in fact a classical cuspidal modular form.
\end{exem}

\bibliographystyle{smfplain}

\begin{thebibliography}{10}

\bibitem{buzzardeigenvarieties}
{\scshape K.~Buzzard} -- {\og Eigenvarieties\fg}, in \emph{{$L$}-functions and
  {G}alois representations}, London Math. Soc. Lecture Note Ser., vol. 320,
  Cambridge Univ. Press, Cambridge, 2007, p.~59--120.

\bibitem{chenevier}
{\scshape G.~Chenevier} -- {\og Une correspondance de {J}acquet-{L}anglands
  {$p$}-adique\fg}, \emph{Duke Math. J.} \textbf{126} (2005), no.~1,
  p.~161--194.

\bibitem{cipra}
{\scshape B.~A. Cipra} -- {\og On the {N}iwa-{S}hintani theta-kernel lifting of
  modular forms\fg}, \emph{Nagoya Math. J.} \textbf{91} (1983), p.~49--117.

\bibitem{colmaz}
{\scshape R.~Coleman {\normalfont \smfandname} B.~Mazur} -- {\og The
  eigencurve\fg}, in \emph{Galois representations in arithmetic algebraic
  geometry (Durham, 1996)}, London Math. Soc. Lecture Note Ser., vol. 254,
  Cambridge Univ. Press, Cambridge, 1998, p.~1--113.

\bibitem{conradirredcpnts}
{\scshape B.~Conrad} -- {\og Irreducible components of rigid spaces\fg},
  \emph{Ann. Inst. Fourier (Grenoble)} \textbf{49} (1999), no.~2, p.~473--541.

\bibitem{conradmodrig}
\bysame , {\og Modular curves and rigid-analytic spaces\fg}, \emph{Pure Appl.
  Math. Q.} \textbf{2} (2006), no.~1, p.~29--110.

\bibitem{kassaei}
{\scshape P.~L. Kassaei} -- {\og A gluing lemma and overconvergent modular
  forms\fg}, \emph{Duke Math. J.} \textbf{132} (2006), no.~3, p.~509--529.

\bibitem{niwa}
{\scshape S.~Niwa} -- {\og Modular forms of half integral weight and the
  integral of certain theta-functions\fg}, \emph{Nagoya Math. J.} \textbf{56}
  (1975), p.~147--161.

\bibitem{mfhi}
{\scshape N.~Ramsey} -- {\og Geometric and {$p$}-adic modular forms of
  half-integral weight\fg}, \emph{Ann. Inst. Fourier (Grenoble)} \textbf{56}
  (2006), no.~3, p.~599--624.

\bibitem{hieigencurve}
\bysame , {\og The half-integral weight eigencurve\fg}, \emph{Algebra Number
  Theory} \textbf{2} (2008), no.~7, p.~755--808.

\bibitem{shimura}
{\scshape G.~Shimura} -- {\og On modular forms of half integral weight\fg},
  \emph{Ann. of Math. (2)} \textbf{97} (1973), p.~440--481.

\end{thebibliography}

\end{document}